\documentclass{article}[12pt]
\usepackage[english]{babel}
\usepackage{amsmath,amssymb,graphicx,hyperref,latexsym,theorem,array}

\everymath={\displaystyle}

\newcommand{\assign}{:=}
\newcommand{\backassign}{=:}
\newcommand{\dueto}[1]{\textup{\textbf{(#1) }}}
\newcommand{\mathD}{\mathrm{D}}
\newcommand{\mathd}{\mathrm{d}}
\newcommand{\nin}{\not\in}
\newcommand{\tmem}[1]{{\em #1\/}}
\newcommand{\tmmathbf}[1]{\ensuremath{\boldsymbol{#1}}}
\newcommand{\tmname}[1]{\textsc{#1}}
\newcommand{\tmnote}[1]{\thanks{\textit{Note:} #1}}
\newcommand{\tmop}[1]{\ensuremath{\operatorname{#1}}}
\newcommand{\tmtextit}[1]{\text{{\itshape{#1}}}}
\newenvironment{proof}{\noindent\textbf{Proof\ }}{\hspace*{\fill}$\Box$\medskip}
\newenvironment{proof*}[1]{\noindent\textbf{#1\ }}{\hspace*{\fill}$\Box$\medskip}
\newtheorem{definition}{Definition}
\newtheorem{lemma}{Lemma}
{\theorembodyfont{\rmfamily}\newtheorem{remark}{Remark}}
\newtheorem{theorem}{Theorem}
\begin{document}

\title{The dyadic Riesz vector II}

\author{
  Komla Domelevo
  \and
  Stefanie Petermichl
  \tmnote{The second author is supported by the European Research Council
  project {CHRiSHarMa no. DLV-682402} and by the Alexander von Humboldt
  Foundation}
}

\maketitle

\begin{abstract}
  We derive a dyadic model operator for the Riesz vector. We show linear upper
  $L^p$ bounds for $1 < p < \infty$ between this model operator and the Riesz
  vector, when applied to functions with values in Banach spaces. By an upper
  bound we mean that the boundedness of the dyadic Riesz vector implies the
  boundedness of the Riesz vector. The same holds for single dyadic Riesz
  transforms and their continuous counterparts. The linear dependence is with
  constant one.
\end{abstract}

\section{Introduction}

Let $X$ be a UMD Banach space and $f \in L^p (\mathbb{R}^d, X)$ be a $L^p$
integrable $X$--valued function, that is
\begin{equation}
  \| f \|_{L^p (\mathbb{R}^d, X)} \assign \left( \int_{\mathbb{R}^d} | f |_X^p
  \mathd x \right)^{1 / p} < \infty . \label{eq:deterministic Lp norm}
\end{equation}
We are interested in comparing the $L^p$ operator norm the Riesz transforms
$\mathcal{R}_i$, $i \in [1, d]$ with their dyadic counterparts, the so--called
dyadic Riesz transforms $\mathcal{S}_i$, $i \in [1, d]$ defined below. In the
case $n = 1$, the authors have recently proved in [arxiv] that the Hilbert
transform $\mathcal{H}$ on the unit disc and the so--called dyadic Hilbert
transform $\mathcal{S}$ have operator norms in linear dependence. The natural
question arising at that point is whether or not the Riesz transforms and
dyadic Riesz transforms have comparable $L^p$ norms. We note further
$\overrightarrow{\mathcal{R}} f \assign (\mathcal{R}_1 f, \ldots,
\mathcal{R}_d f)$ and $\overrightarrow{\mathcal{S}} f \assign (\mathcal{S}_1
f, \ldots, \mathcal{S}_d f)$ respectively the Riesz and dyadic Riesz vectors.
Their $L^p$ norms are
\[ \| \overrightarrow{\mathcal{R}} f \|_p \assign \left( \int_{\mathbb{R}^d} |
   \overrightarrow{\mathcal{R}} f |_{\ell^2}^p \mathd x \right)^{1 / p}, \quad
   \| \overrightarrow{\mathcal{S}} f \|_p \assign \left( \int_{\mathbb{R}^d} |
   \overrightarrow{\mathcal{S}} f |_{\ell^2}^p \mathd x \right)^{1 / p} \]
with $| \vec{u} |_{\ell^2} \assign \left( \sum_{i = 1}^d | u_i |_X^2
\right)^{1 / 2}$. In the companion paper [arxiv], the authors prove linear
lower bounds
\[ \forall i \in [1, d], \quad \| \mathcal{S}_i \|_{p \rightarrow p} \leqslant
   c_0 \| \mathcal{R}_i \|_{p \rightarrow p}, \]
as well as
\[ \| \overrightarrow{\mathcal{S}} \|_{p \rightarrow p} \leqslant c_0 \|
   \overrightarrow{\mathcal{R}} \|_{p \rightarrow p}, \]
for a universal constant $c_0$. The main goal of the present paper is to prove
the dimensionless upper bounds

\begin{theorem}[Upper bound for a single Riesz transform]
  \label{T: upper bound single Riesz}We have for all $i \in [1, d]$,
  \[ \| \mathcal{R}_i \|_{p \rightarrow p} \leqslant \| \mathcal{S}_i \|_{p
     \rightarrow p} . \]
\end{theorem}

\begin{theorem}[Upper bound for the Riesz vector]
  \label{T: upper bound Riesz vector}We have
  \[ \| \overrightarrow{\mathcal{R}} \|_{p \rightarrow p} \leqslant \|
     \overrightarrow{\mathcal{S}} \|_{p \rightarrow p} . \]
\end{theorem}

As compared to the onedimensional case studied in {\cite{DomPet2023a}}, \
several new difficulties arise, that we detail later. To summarize, (i) the
Riesz transforms are directly defined as stochastic integrals but as
conditional expectations upon arrival, (ii) convergence of weak formulations
are needed, (iii) the domain and its boundary are unbounded, (iv) the Riesz
vector estimate above requires a specific weak formulation, and finally (v)
the discrete random walks associated to the different dyadic Riesz transforms
are more intricate.

We introduce now the dyadic setting and define the dyadic Riesz transforms.

\paragraph{Dyadic system}Let $I_0 = [0, 1)$ and $\mathcal{D}$ the dyadic
intervals
\[ \mathcal{D} \assign \left\{ [m 2^{- i}, (m + 1) 2^{- i}) ; \quad 0
   \leqslant i, 0 \leqslant m \leqslant 2^i - 1 \right\} = \bigcup_{i
   \geqslant 0} \mathcal{D}_i \]
where $\mathcal{D}_i \assign \left\{ [m 2^{- i}, (m + 1) 2^{- i}) ; \quad 0
\leqslant m \leqslant 2^i - 1 \right\}$ is the set of dyadic intervals of the
$i$--th generation having therefore length $2^{- i}$. We note further
$\hat{I}$ the parent of $I$, and $I_-$ (resp. $I_+$) the left (resp. right)
children of $I$, $\mathcal{D}^- \subset \mathcal{D}$ the subset of left
children and $\mathcal{D}^+ \subset \mathcal{D}$ the subset of right children.
We have therefore $\mathcal{D} = \{ I_0 \} \cup \mathcal{D}^- \cup
\mathcal{D}^+$. We also note $\mathcal{D}_i^- \assign \mathcal{D}_i \cap
\mathcal{D}^-$ the left children of $\mathcal{D}_i$, and $\mathcal{D}_i^+
\assign \mathcal{D}_i \cap \mathcal{D}^+$ the right children of
$\mathcal{D}_i$.

Let the $X$--valued function $f$ locally integrable on $I_0$. Its Haar
decomposition writes.
\[ f (x) = \langle f \rangle_{I_0} + \sum_{I \in \mathcal{D}} (f, h_I) h_I, \]
where $h_I \assign (\tmmathbf{1}_{I_+} -\tmmathbf{1}_{I_-}) / \sqrt{| I |}$ is
the $L^2$--normalized Haar function associated to the interval $I$, and
$\langle f \rangle_I \assign (f, \tmmathbf{1}_I) / | I |$ is the average of
$f$ on $I$. The dyadic Hilbert transform $\mathcal{S}$ studied in
{\cite{DomPet2022a,DomPet2023a}} is the operator sending
\[ \langle f \rangle_{I_0} \mapsto 0, \quad h_{I_0} \mapsto 0, \quad
   h_{I_{\pm}} \mapsto \pm h_{\mp} . \]
Writing the Haar decomposition of $f$ in the equivalent forms

\[ \begin{array}{lll}
     f & = & \langle f \rangle_{I_0} + \sum_{I \in \mathcal{D}} (f, h_I) h_I\\
     & = & \langle f \rangle_{I_0} + (f, h_{I_0}) h_{I_0} + \sum_{k =
     1}^{\infty} \sum_{I \in {\mathcal{D}_k} } (f, h_I) h_I\\
     & = & \langle f \rangle_{I_0} + (f, h_{I_0}) h_{I_0} + \sum_{k =
     0}^{\infty} \sum_{I \in {\mathcal{D}_k} } (f, h_{I_-}) h_{I_-} + (f,
     h_{I_+}) h_{I_+},
   \end{array} \]
yields the Haar decomposition of the dyadic Hilbert transform $\mathcal{S} f$
applied to $f$
\[ \begin{array}{lll}
     \mathcal{S} f & = & \sum_{k = 1}^{\infty} \sum_{I \in {\mathcal{D}_k} }
     (f, h_I) \mathcal{S} h_I\\
     & = & \sum_{k = 0}^{\infty} \sum_{I \in {\mathcal{D}_k} } - (f, h_{I_-})
     h_{I_+} + (f, h_{I_+}) h_{I_-} .
   \end{array} \]

\paragraph{Dyadic Riesz transforms}We now generalize the construction of the
dyadic Hilbert transfrom above to higher dimensional settings. Let
$\mathcal{D}_i = \{ I \in \mathcal{D} ; | I | = 2^{- i} | I_0 | \}$ the
$i$--th generation of dyadic intervalls, and $\mathcal{D}^{(i)} = \cup_{k =
0}^{\infty} \mathcal{D}_{k d + i}$ the $i$--th slice of the dyadic system,
that is all dyadic intervals whose generation number equals $i$ modulo $d$.
The operator $\mathcal{S}_i$, $1 \leqslant i \leqslant d$, is defined as
\[ \mathcal{S}_i : \quad h_I \mapsto \left\{ \begin{array}{ll}
     \mathcal{S} h_I & , I \in \mathcal{D}^{(i)}\\
     0 & , I \nin \mathcal{D}^{(i)} .
   \end{array} \right. \]
In other words it acts on the $i$--th slice in the same way as the dyadic
Hilbert transform does. However, it sends to zero all elements of the other $d
- 1$ slices. Correspondingly, the Haar decomposition of $f$, written as blocks
of $d$ generations, reads
\[ \begin{array}{lll}
     f & = & \langle f \rangle_{I_0} + (f, h_{I_0}) h_{I_0} + \sum_{k =
     0}^{\infty} \sum_{j = 1}^d \sum_{I \in {\mathcal{D}_{k d + j}} } (f, h_I)
     h_I\\
     & = & \langle f \rangle_{I_0} + (f, h_{I_0}) h_{I_0} + \sum_{k =
     0}^{\infty} \sum_{j = 1}^d \sum_{I \in {\mathcal{D}_{k d + j - 1}} } (f,
     h_{I_-}) h_{I_-} + (f, h_{I_+}) h_{I_+},
   \end{array} \]
and the $i$--th dyadic Riesz transform $\mathcal{S}_i f$ defined above writes
as

\[ \mathcal{S}_i f = \sum_{I \in \mathcal{D}^{(i)}} (f, h_I)  \mathcal{S} h_I
   = \sum_{k = 0}^{\infty} \sum_{I \in \mathcal{D}_{k d + i - 1}} - (f,
   h_{I_-}) h_{I_+} + (f, h_{I_+}) h_{I_-} . \]

Inspired by the onedimensional case, we wish to approximate the Riesz
transforms in $\mathbb{R}^d$ with the help of $(d + 1)$--dimensional discrete
random walks built upon the dyadic filtration. However, a dyadic tree is
naturally suited for coding a onedimensional discrete random walk consisting
of random $\pm 1$ steps, i.e. left/right steps. For our task, we need to use
the steps provided by the dyadic tree in order to produce steps along $(d +
1)$ different directions. In the case $d = 1$, this was done by assigning \
left children of the dyadic tree to horizontal steps of the twodimensional
discrete random walk, and right children to vertical steps. For the general
case, we need different assignments and interpretations of the tosses offered
by the dyadic tree.

Let us first review the random objects involved in the continuous setting.

\paragraph{Stochastic representations of Riesz transforms}The stochastic
representation of Riesz transforms in $\mathbb{R}^d$ follows the work of
{\tmname{Gundy--Varopoulos}} {\cite{GunVar1979a}}, where the Riesz transforms
are conditional expectations of suitable martingales transforms. See also
{\tmname{Ba{\~n}uelos}} and {\tmname{Wang}} {\cite{BanWan1996}}. Let $\Omega
\assign \mathbb{R}^{d + 1}_+ = \{ (x_0, x_1, \ldots, x_d) ; x_0 > 0, x = (x_1,
\ldots, x_d) \in \mathbb{R}^d \}$ the upper--half space. We call $x_0 \in
\mathbb{R}_+$ the ``vertical'' component of $x \in \Omega$. Given $y > 0$, let
$W^y = (W^{y, 0}, W^{y, 1}, \ldots, W^{y, d})$ the $(d + 1)$--dimensional
Brownian motion started at time $0$ at the point $(y, 0, \ldots, 0) \in
\Omega$ above the origin. Let $\tau$ the stopping time, almost surely finite,
defined as
\[ \tau \assign \inf \{ t > 0 ; W_t \nin \Omega \} . \]
This is also the hitting time of $W^y$ hitting the boundary $\partial \Omega
=\mathbb{R}^d$ of the upper--half space. We further denote by $(W_t^{y,
\tau})_{t \geqslant 0} \assign (W^y_{t \wedge \tau})_{t \geqslant 0}$ the
stopped process. By construction $W^{\tau, y}_t \in \Omega$ for all $t > 0$
and $W^{\tau, y}_t \in \partial \Omega$ for all $t \geqslant \tau$. In
particular $W^{\tau, y}_{\infty} \in \partial \Omega$.

Let now $f \in L^p (\partial \Omega, X)$ a smooth function (see the precise
definition of smooth below at the end of the section). We also denote by $f :
\Omega \rightarrow X$ its harmonic extension in the upper--half space. Given
$y > 0$ and $f$ as above, Ito formula ensures that the stochastic process
$\mathcal{M}_t^{y, f} \assign f (W^{y, \tau}_t)$ is a martingale satisfying
\[ \mathcal{M}_t^{y, f} \assign f (W^{y, \tau}_t) = f (y) + \int_0^t \nabla f
   (W^{y, \tau}_t) \cdot \mathd W^{y, \tau}_t, \]
and its $i$--th martingale transform $\mathcal{M}_t^i$ is defined as
\begin{equation}
  \mathcal{M}_t^{y, i} \assign \int_0^t (A_i \nabla f) (W^{y, \tau}_t) \cdot
  \mathd W^{y, \tau}_t = \int_0^t \nabla f (W^{y, \tau}_t) \cdot (A_i^{\ast}
  \mathd W^{y, \tau}_t), \quad i \in [1, d], \label{eq:Mif}
\end{equation}
where $A_i \in \mathbb{R}^{(d + 1) \times (d + 1)}$ is the matrix defining the
$i$--th martingale transform. Namely if $(e_0, e_1, \ldots, e_d)$ denotes the
standard basis of $\mathbb{R}^{d + 1}$, we have set
\[ A_i e_0 = - e_i, \quad A_i e_i = e_0, \quad A_i e_j = 0, \quad j \in [1,
   d], j \neq i, \]
and we observe that its transpose $A_i^{\ast}$ obeys
\[ A^{\ast}_i e_0 = e_i, \quad A_i^{\ast} e_i = - e_0, \quad A^{\ast}_i e_j =
   0, \quad j \in [1, d], j \neq i. \]
To summarize,
\begin{equation}
  A_i \nabla f = (- \partial_i f, 0, \ldots, \partial_0 f, 0, \ldots, 0),
  \label{eq:Cauchy Riemann on grad f}
\end{equation}
and for the transposed counterpart, where we let the transform act on the
random walk,
\begin{equation}
  A_i^{\ast} \mathd W_t = (\mathd W^i_t, 0, \ldots, 0, - \mathd W^0_t, 0,
  \ldots, 0) . \label{eq:Cauchy Riemann on dW}
\end{equation}
Following {\cite{GunVar1979a}}, see also {\cite{BanWan1996}}, the weak
formulation for the $i$--th Riesz transforms writes
\begin{equation}
  (\mathcal{R}_i f, g) =\mathbb{E} \int_0^{\infty} (A_i \nabla f) (W^{y,
  \tau}_t) \cdot \nabla g (W^{y, \tau}_t) \mathd t, \label{eq:weak formulation
  Ri}
\end{equation}
for any $X^{\ast}$--valued harmonic test function $g : \Omega \rightarrow
X^{\ast}$ with smooth boundary values $g \in L^q (\partial \Omega, X^{\ast})$.
Equivalently, the $i$--th Riesz transform is recovered pointwise thanks to the
conditional expectation (see {\cite{GunVar1979a}}),
\begin{equation}
  \forall x \in \partial \Omega, \quad \mathcal{R}_i f (x) = \lim_{y
  \rightarrow + \infty} \mathbb{E} (\mathcal{M}_{\infty}^{y, i} | W^{y,
  \tau}_{\infty} = (0, x)) . \label{eq:conditonal expectation Ri}
\end{equation}

\paragraph{Summary of notations and functional setting}Recall that we denote
$\Omega =\mathbb{R}^{d + 1}_+$ the upper half space and $\partial \Omega
=\mathbb{R}^d$ its boundary. We assume without loss of generality that $f \in
L^p (\partial \Omega, X)$ and its harmonic extension (also noted $f$) $f \in
L^p (\Omega, X)$ are smooth Frechet differentiable functions, that is $f \in
\mathcal{C}^k$ for all $k \geqslant 0$. We note $| f (x) | \assign | f (x)
|_X$ the Banach space norm of $f (x)$. The $L^p$--norm of $f$ is defined as in
\eqref{eq:deterministic Lp norm}.

Further if $Y$ is an $L^p$ integrable random variable, the stochastic $L^p$
norm is defined as
\[ \| Y \|_p \assign \mathbb{E} (| Y |^p)^{1 / p} . \]
If $Y$ is real valued, then $| Y |$ denotes the absolute value of $Y$ whereas
if $Y$ is $X$--valued, then $| Y | \assign | Y |_X$ denotes the Banach space
norm of $Y$. For a real random variable, we note its variance $\mathbb{V} (Y)
\assign \mathbb{E} (Y -\mathbb{E} (Y))^2$.

If $f \assign f (x_1, \ldots, x_m)$ is a $X$--valued function of $m$ variables
defined on the open set $U \subset \mathbb{R}^m$, we note $\mathD f \assign
(\partial_1 f, \ldots, \partial_m f)$ its derivatives in the Frechet sense,
where $\mathD f : U \times \mathbb{R}^m \rightarrow X$ is continuous. Further,
given a $m$--multiindex $\alpha \assign (\alpha_1, \alpha_2, \ldots,
\alpha_m)$ with $| \alpha | = k$ we note as usual $\mathD^{\alpha} f \assign
\partial_1^{\alpha_1} \ldots \partial_m^{\alpha_m} f$ its partial derivatives
of order $k$ in the Frechet sense, where $\mathD^{\alpha} f : U \times
(\mathbb{R}^m)^k \rightarrow X$ is continuous. Finally,using again
multiindices, monomials of the from $x^{\alpha}$, where $x \assign (x_1,
\ldots, x_m)$, are a shorthand for $x^{\alpha} \assign x_1^{\alpha_1} \ldots
x_m^{\alpha_m}$.

The dual space of $X$ is noted $X^{\ast}$, and the dual space of $L^p (U, X)$
is noted $L^q (U, X^{\ast})$ with $1 / p + 1 / q = 1$.

For the convergence results, our main parameters are $\varepsilon > 0$, a
small number, and $T > 0$, a large number. We note $c (\varepsilon)$ a generic
function tending to $0$ uniformly in $T$ when $\varepsilon$ goes to zero, $c_T
(\varepsilon)$ a generic function tending to $0$ when $\varepsilon$ goes to
zero for any fixed $T$, $c (T)$ a generic function tending to $0$ when $T$
goes to infinity, uniformely in $\varepsilon$. It is implicit that those
functions all depend on the fixed function $f$ and its derivatives. Further
dependences will be mentionned when needed.

\section{Discrete random walks and martingales}

Following the strategy used for the dyadic Hilbert transform, we wish to
approximate continuous martingales such as $\mathcal{M}_t^f$ that are driven
by $W^{y, \tau}_t$ with discrete martingales $M_k^f$ driven by a $(d +
1)$--dimensional discrete random walk $(B_k)_{k \in \mathbb{N}}$ living on a
{\tmem{dyadic}} filtered probability space. At difference with the Hilbert
case, the construction of a single discrete random walk does not allow us to
estimate the different Riesz transforms. We build instead a set of discrete
random walks $(B^{(i)}_k)_{k \in \mathbb{N}}$ for $i \in [1, d]$, each of
which is adapted to the study of one single Riesz transform. For a given $i
\in [1, d]$, we are aiming at mimicking the action of the $i$--th Riesz
martingale transform $\mathd W_t \rightarrow A_i^{\ast} \mathd W_t$ as seen in
the expression \eqref{eq:Mif} of $\mathcal{M}^i_t$ and in \eqref{eq:Cauchy
Riemann on dW}. In other words, we want
\[ \mathcal{S}_i \mathd B_k^{(i)} = A_i^{\ast} \mathd B_k^{(i)}, \quad
   \tmop{all} k \geqslant 1, \]
where $\mathd B_k^{(i)} \assign B_k^{(i)} - B_{k - 1}^{(i)}$ is the $k$--th
increment of the discrete random walk $B^{(i)}$ and $\mathcal{S}_i$ the
$i$--th dyadic Riesz transform defined earlier. We set the starting point of
$B^{(i)}$ to coincide with that of $W^{y, \tau}$, namely
\[ B_0^{(i)} = W_0^{y, \tau} = (y, 0, \ldots, 0), \]
and we will omit the superscript $y$ for the discrete random walks. Let us
split $B^{(i)}$ into its vertical component $B^{(i), 0} \in \mathbb{R}^+$, a
one--dimensional discrete random walk started at $y > 0$, and its horizontal
component $B^H \in \mathbb{R}^d$,
\[ B_k^{(i)} = (B_k^{(i), 0}, B^H_k), \quad \tmop{all} k \geqslant 0, \]
where $B^H_k = (B^{H, 1}_k, B^{H, 2}_k, \ldots, B^{H, d}_k)$ is a
$d$--dimensional discrete random walk started at $0 \in \mathbb{R}^d$.

\

We associate nodes of the dyadic tree to certain increments of $B^{(i)}$. We
start with the horizontal part. Here, for any slice $\mathcal{D}^{(j)}$, $j
\in [1, d]$, we associate each left interval of that slice to an increment of
$B^H_k$ in the $j$--th direction. Schematically, for $d = 2$, those are the
nodes and directions involved, marked with $x_1$ for producing an increment of
$B^H$ parallel to $e_1$ and marked with $x_2$ for producing an increment of
$B^H$ parallel to $e_2$:

\begin{figure}[h] \centering
  \raisebox{-0.966367476519199\height}{\includegraphics[width=8.18711875245966cm,height=4.66408402203857cm]{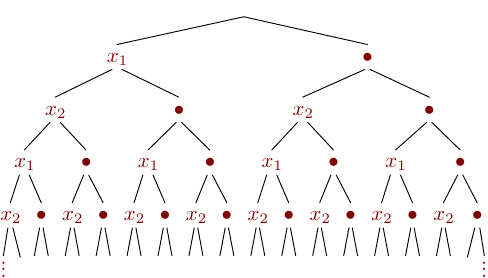}}
  \caption{Nodes contributing to $B^H$}
\end{figure}

Now the definition of the vertical component $B^{(i), 0}$ depends on the
direction $i$ of the chosen dyadic Riesz transform $\mathcal{S}_i$ under
study. Given $i \in [1, d]$, we will collect increments of $B^{(i), 0}$ only
on intervals that are on the one hand right children \tmtextit{and} on the
other hand that belong to the $i$--th slice $\mathcal{D}^{(i)}$.
Schematically, for $d = 2$, those are the nodes involved respectively for
$B^{(1), 0}$ and $B^{(2), 0}$:

\begin{figure}[h] \centering
  \raisebox{-0.955922865013774\height}{\includegraphics[width=6.30836776859504cm,height=3.71428571428571cm]{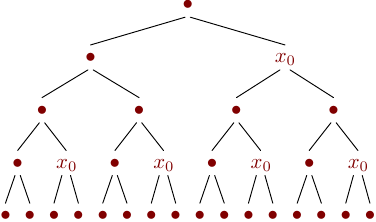}}
  \caption{Nodes involved for $B^{(1), 0}$}
\end{figure}

\begin{figure}[h] \centering
  \raisebox{-0.956546859342001\height}{\includegraphics[width=7.62399153876427cm,height=3.76762347500984cm]{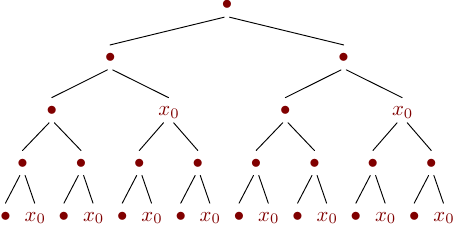}}
  \caption{Nodes involved for $B^{(2), 0}$}
\end{figure}

The construction of all these objects obeys a period equal to $d$ generations.
Let us note $\mathcal{L}_k \assign \cup_{j = 1}^d \mathcal{D}_{(k - 1) d +
j}$, $k \geqslant 1$, the $k$--th layer of the dyadic tree, made of $d$
successive generations of intervals. Notice that the first layer starts at the
first generation $\mathcal{D}_1$ of dyadic intervals, i.e. $\mathcal{L}_1 =
\mathcal{D}_1 \cup \ldots \cup \mathcal{D}_d$. The root $\mathcal{D}_0 = \{
I_0 \}$ does not belong to any layer. This is because we only consider
intervals with siblings. Introduce the tosses $\varepsilon_I (x) = \pm 1$
defined as the $L^{\infty}$ normalized Haar functions
\[ \varepsilon_I (x) \assign \sqrt{| I |} h_I (x) =\tmmathbf{1}_{I_+} (x)
   -\tmmathbf{1}_{I_-} (x) = \left\{ \begin{array}{ll}
     1, & x \in I_-\\
     - 1, & x \in I_+\\
     0, & \tmop{otherwise}
   \end{array} \right. \]
We now define tosses parametrized with generation and left/right position. For
the generation $i = 0$ we have simply
\[ \varepsilon_0 (x) = \varepsilon_{I_0} (x) . \]
Similarly the $i$--th toss can be defined as
\[ \varepsilon_i (x) = \sum_{I \in \mathcal{D}^i} \varepsilon_I (x) . \]
Notice that only one term is non zero in the sum above, namely $\varepsilon_i
(x) = \varepsilon_{I^x_i} (x)$ where $I^x_i$ denotes the unique dyadic
interval in $\mathcal{D}_i$ that contains $x$. For our purpose, we need to
split those tosses into two groups,
\[ \varepsilon_i^- (x) = \left\{ \begin{array}{ll}
     \varepsilon_i (x), & x \in \mathcal{D}_i^-\\
     0, & x \in \mathcal{D}_i^+
   \end{array} \right., \quad \varepsilon_i^+ (x) = \left\{ \begin{array}{ll}
     0, & x \in \mathcal{D}_i^-\\
     \varepsilon_i (x), & x \in \mathcal{D}_i^+
   \end{array} \right., \]
or equivalently
\[ \varepsilon^{\pm}_i (x) =\tmmathbf{1} (\varepsilon_{i - 1} (x) = \pm 1)
   \varepsilon_i (x) = \sum_{I \in \mathcal{D}_i^{\pm}} \varepsilon_I (x) . \]
Here is a summary of the different tosses carried by the dyadic tree up to
generation $\mathcal{D}_4$.

\

\begin{figure}[h] \centering
  \raisebox{-0.958652606730754\height}{\includegraphics[width=9.6448494687131cm,height=3.95950167256985cm]{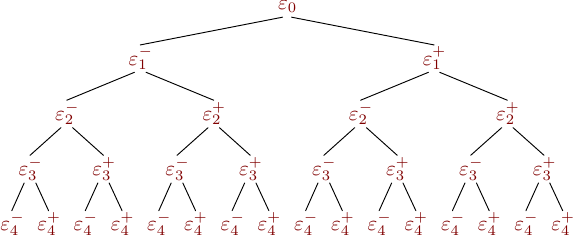}}
  \caption{Tosses on the dyadic tree.}
\end{figure}

Let $\delta$ a small time step. For each $k \geqslant 1$, we construct the
step $\mathd B_k^{(i)}$ of the discrete random walk $B^{(i)}_k$ using
information from the layer $\mathcal{L}_k$:
\[ \mathd B_k^H (x) = \sqrt{2 \delta}  (\varepsilon_{(k - 1) d + 1}^- (x),
   \varepsilon_{(k - 1) d + 2}^- (x), \ldots, \varepsilon_{(k - 1) d + d}^-
   (x)), \quad \mathd B_k^{(i), 0} (x) = \sqrt{2 \delta} \varepsilon_{(k - 1)
   d + i}^+ (x) . \]
It follows
\begin{eqnarray}
  \mathcal{S}_i \mathd B^{(i)}_k & = & \mathcal{S}_i (\mathd B^{(i), 0}_k,
  \mathd B_k^H) \nonumber\\
  & = & \sqrt{2 \delta}  \mathcal{S}_i (\varepsilon_{(k - 1) d + i}^+,
  \varepsilon_{(k - 1) d + 1}^-, \varepsilon_{(k - 1) d + 2}^-, \ldots,
  \varepsilon_{(k - 1) d + d}^-) \nonumber\\
  & = & \sqrt{2 \delta}  (\varepsilon^-_{(k - 1) d + i}, 0, \ldots, 0, -
  \varepsilon_{(k - 1) d + i}^+, 0, \ldots, 0) \nonumber\\
  & = & A_i^{\ast} \mathd B^{(i)}_k .  \label{eq:SidB}
\end{eqnarray}
This achieves our first goal of expressing the action of $\mathcal{S}_i$ as a
generalized Cauchy--Riemann relation on the dyadic system.

\

\paragraph{Coarse graining and stopping times}Let $T > 0$ a fixed time, and
$N \in \mathbb{N}$ large. Define the time step $\delta$ through $T = N^5
\delta$. In order to denote the corresponding discrete times, we will use the
indices $k, l \in [0, N^5]$, typically $t_k \assign k \delta$, $t_l \assign l
\delta$. We introduce a coarser time-step $\theta$ defined as $\theta \assign
N \delta$. Notice that $T = N^4 \theta$, so that $\theta$ tends to zero as $N$
goes to infinity for fixed $T$. The discrete times corresponding to this
time--step will use indices $n, m \in [0, N^4]$, typically $t_n \assign n
\theta$, $t_m \assign m \theta$. For a given $i \in [1, d]$, we define the
discrete random walk $X^{(i)}$ by sampling $B^{(i)}$ at times that are
multiples of $\theta = N \delta$, therefore with indices that are multiples of
$N$,
\[ \forall n \geqslant 0, \quad X^{(i)}_n \assign B^{(i)}_{n N} . \]
Notice further that, for all $n \geqslant 1$,
\[ \mathd X^{(i)}_n \assign X^{(i)}_n - X^{(i)}_{n - 1} = B^{(i)}_{n N} -
   B^{(i)}_{(n - 1) N} = \sum_{l = 1}^N \mathd B^{(i)}_{(n - 1) N + l} . \]
In order to stop $X^{(i)}$ before it leaves the upper--half space, we set
$\varepsilon = 1 / N$ and introduce
\[ \Omega_{\varepsilon} \assign \left\{ x = (x_0, x_1, \ldots, x_d) \in \Omega
   ; \quad x_0 > \varepsilon \right\}, \]
as well as the stopping time, for each $i \in [1, d]$,
\[ \tau_{\varepsilon}^{(i)} \assign \inf \left\{ t_n ; \quad X^{(i)}_n \nin
   \Omega_{\varepsilon} \right\} . \]
We will denote by $n_{\varepsilon}^{(i)}$ the corresponding random index such
that $\tau_{\varepsilon}^{(i)} \assign n_{\varepsilon}^{(i)} \theta$. One
observes that for small enough $\varepsilon > 0$ (i.e. large enough $N$), we
have for all $n$, $| \mathd X^{(i)}_n | \leqslant N \delta \leqslant T N^{- 4}
< \varepsilon / 2$. Therefore since $X_{n_{\varepsilon}^{(i)} - 1} \in
\Omega_{\varepsilon}$ we have that $X_{n_{\varepsilon}^{(i)}}$ is in the band
$\Omega \backslash \Omega_{\varepsilon}$. Equivalently,
$\tau^{(i)}_{\varepsilon}$ is the the first time $t_n$ at which $X^{(i)}$ is
at a distance less or equal to $\varepsilon$ from the boundary $\partial
\Omega$ of the upper--half space. This stopping time is also a stopping time
for $B^{(i)}$. We note $B^{(i), \tau_{\varepsilon}^{(i)}}$ the corresponding
stopped process. Notice that $B^{(i), \tau_{\varepsilon}^{(i)}} \in
\Omega_{\varepsilon / 2}$ for all times. In all what follows, all random
walks, continuous or discrete, are meant as stopped process. For convenience
we might write $X^{(i)}_n$ instead of $X^{(i), \tau^{(i)}_{\varepsilon}}_n$,
or $B^{(i)}_k$ instead of $B^{(i), \tau^{(i)}_{\varepsilon}}_k$. By
construction we have always $W_t^y, X_n^{(i)}, B_k^{(i)} \in \Omega$.

\

\paragraph{Discrete martingales and their transforms}Given $f \in L^p
(\mathbb{R}^d, X)$ and its harmonic extension also noted $f$, we now aim at
approximating both $\mathcal{M}_t^f$ and $\mathcal{M}^i_f$ by discrete
martingales adapted to the action of the dyadic Riez transforms. An important
difference with the onedimensional case is that the discretization of not only
$\mathcal{M}^i_f$ but also $\mathcal{M}^f$ depends on the single Riesz
transform under study. For each $i \in [1, d]$, we define the approximate
$M^{(i), f}$ of $\mathcal{M}^f$ as
\[ M^{(i), f}_k \assign \mathcal{M}_0^f + \sum_{\ell = 1}^k \nabla f (B_{k -
   1}^{(i)}) \cdot \mathd B_k^{(i)}, \]
and the approximate $M^{(i), i}$ of $\mathcal{M}^i_t$ as
\[ M_k^{(i), i} \assign \sum_{\ell = 1}^k A_i \nabla f (B_{k - 1}^{(i)}) \cdot
   \mathd B_k^{(i)} . \]
The superscript $(i)$ refers to the fact that the discrete process is driven
by the $i$--th discrete random walk $B^{(i)}$, whereas the superscript $i$
refers to the fact that we are approximating the $i$--th martingale transform
$\mathcal{M}^i$. Our first goal is achieved:

\

\begin{lemma}
  \label{L: Lp estimate for Mii}For all $f$, all $i \in [1, d]$, we have
  $M_k^{(i), i} = \mathcal{S}_i M^{(i), f}_k$ and therefore
  \begin{equation}
    \| M_k^{(i), i} \|_p \leqslant \| \mathcal{S}_i \|_{p \rightarrow p}  \|
    M^{(i), f}_k \|_p . \label{eq:Lp bound for discrete martingales}
  \end{equation}
\end{lemma}

\begin{proof}
  By construction of the discrete random walks and property \eqref{eq:SidB},
  we have immediately
  \[ \begin{array}{lll}
       \mathcal{S}_i M^{(i), f}_k & = & \sum_{\ell = 1}^k \nabla f (B_{k -
       1}^{(i)}) \cdot \mathcal{S}_i \mathd B_k^{(i)} = \sum_{\ell = 1}^k
       \nabla f (B_{k - 1}^{(i)}) \cdot A_i^{\ast} \mathd B_k^{(i)}\\
       & = & \sum_{\ell = 1}^k A_i \nabla f (B_{k - 1}^{(i)}) \cdot \mathd
       B_k^{(i)} = M_k^{(i), i}
     \end{array} \]
  and the result follows.
\end{proof}

\begin{remark}
  If $X^{(i)}_n \assign B^{(i)}_{n N}$ is the coarse--grained random walk, we
  could also define the stochastic integrals based on $X^{(i)}$ instead of
  $B^{(i)}$, i.e.
  \[ M^{(i), f}_n \assign \mathcal{M}_0^f + \sum_{m = 1}^n \nabla f (X_{m -
     1}^{(i)}) \cdot \mathd X_m^{(i)}, \quad M_n^{(i), f_i} \assign \sum_{m =
     1}^n A_i \nabla f (X_{m - 1}^{(i)}) \cdot \mathd X_m^{(i)} . \]
  and we would still have by linearity of $\mathcal{S}_i$ that $M_n^{(i), f_i}
  = \mathcal{S}_i M^{(i), f}_n$. It is not a big simplification to the proof
  of convergence.
\end{remark}

\

\paragraph{Strategy and plan}Our main goal is now to show that one can pass
to the limit in inequality \eqref{eq:Lp bound for discrete martingales} of
Lemma \ref{L: Lp estimate for Mii}. As compared to the onedimensional case $d
= 1$ studied in {\cite{DomPet2023a}}, a few new difficulties arise.

First, \ the Riesz transforms are not directly defined by a stochastic
integral, but rather through a projection of a stochastic integral. This
projection is the conditional expectation upon arrival as in
\eqref{eq:conditonal expectation Ri}. This projection would not be easy to
handle as a limit of a discrete process. For this reason, part of our proof
relies on the weak formulation \eqref{eq:weak formulation Ri} and related
convergence results from discrete weak formulations to continuous weak
formulations.

Second, the estimate of the Riesz vector requires a specific weak formulation
that we can relate to the dyadic Riesz transforms.

Finally, the upper half space as well as its boundary are unbounded, and the
random processes involved in the stochastic representation need to start at
``infinity'', which involves a supplementary limiting process as compared to
the onedimensional case on the circle.

\

\section{Auxiliairy results}

Recall that all discrete random walks $B^{(i)}$, or $X^{(i)}$for $i \in [1,
d]$ are started at $(y, 0, \ldots, 0) \in \Omega$ and stopped at time
$\tau_{\varepsilon}^{(i)}$. We note $W^{t, x, \tau}$ the $(d +
1)$--dimensional Brownian motion started at time $t$ at point $x \in \Omega$
and stopped at time $\tau$ when hitting $\partial \Omega$. We define,

\begin{definition}
  We say that $X$ is weakly consistent with $W^{\tau}$, iff there exists a
  function $c \assign c (\varepsilon)$ tending to zero when $\varepsilon$ goes
  to zero, such that for all $n$,
  \[ | \mathbb{E} (X_{n + 1} - X_n | \widetilde{\mathcal{F}}_n) -\mathbb{E}
     (W^{t_n, X_n, \tau}_{t_{n + 1}} - W_{t_n}^{t_n, X_n, \tau} |
     \widetilde{\mathcal{F}}_n) | \leqslant \theta c (\varepsilon), \]
  and such that for all $(d + 1)$--multiindex $\alpha \assign \left( \alpha_0,
  \alpha_1 {, \ldots, \alpha_d}  \right)$ with $| \alpha | = 2$, there holds
  \[ | \mathbb{E} ((X_{n + 1} - X_n)^{\alpha} | \widetilde{\mathcal{F}}_n)
     -\mathbb{E} ((W^{t_n, X_n, \tau}_{t_{n + 1}} - W_{t_n}^{t_n, X_n,
     \tau})^{\alpha} | \widetilde{\mathcal{F}}_n) | \leqslant \theta c
     (\varepsilon), \]
\end{definition}

and claim,

\begin{lemma}
  \label{L: weak consistency} Let $i \in [1, d]$. The discrete stopped process
  $X^{(i)}$ is weakly consistent with the continuous stopped process
  $W^{\tau}$.
\end{lemma}

\

\begin{proof}
  This is an easy consequence of the following two lemmas.
\end{proof}

\begin{lemma}[Continuous moments]
  \label{L: continuous moments} Let $W^{t, x, \tau}$ as above. We have
  \[ \forall t \geqslant 0, \forall x \in \Omega, \quad \mathbb{E} (W_{t +
     \theta}^{\tau, t, x} - W_t^{\tau, t, x}) = 0. \]
  Moreover there exists a function $c_T (\varepsilon)$ tending to zero with
  $\varepsilon$ for all fixed $T$, such that for all coordinates $j, k \in [0,
  d]$
  \[ \forall t \geqslant 0, \forall x \in \Omega_{\varepsilon}, \quad
     \mathbb{E} (W_{t + \theta}^{t, x, \tau, j} - W_t^{t, x, \tau, j})^2 =
     \theta (1 + c_T (\varepsilon)), \]
  \[ \forall t \geqslant 0, \forall x \in \Omega_{\varepsilon}, \quad
     \mathbb{E} (W_{t + \theta}^{t, x, \tau, j} - W_t^{t, x, \tau, j}) (W_{t +
     \theta}^{t, x, \tau, k} - W_t^{t, x, \tau, k}) = \theta c_T (\varepsilon)
     \quad \tmop{for} j \neq k. \]
  Finally for all $p \geqslant 2$, there holds for all coordinates $j \in [0,
  d]$
  \[ \forall t \geqslant 0, \forall x \in \Omega, \forall p \geqslant 2, \quad
     \mathbb{E} (| W_{t + \theta}^{\tau, t, x, j} - W_t^{\tau, t, x, j} |^p)
     \lesssim \theta^{p / 2} . \]
\end{lemma}

\begin{proof}
  The proof is the same as in the case $d = 1$. See {\cite{DomPet2023a}}.
\end{proof}

\begin{lemma}[Discrete Moments]
  \label{L: discrete moments}Let $i \in [1, d]$. Let $X^{(i)}$ the $i$--th
  stopped discrete random walk defined above. We have
  \[ \forall n \in \mathbb{N}, \quad \mathbb{E} \left( \mathd X^{(i)}_{n + 1}
     | {\widetilde{\mathcal{F}}_n}  \right) = 0 \]
  Moreover there exists a function $c (\varepsilon)$, tending to zero with
  $\varepsilon$, such that for all coordinates \ $j, k \in [0, d]$
  \[ \forall n \in \mathbb{N}, \quad \mathbb{E} ((\mathd X^{(i), j}_{n + 1})^2
     | \widetilde{\mathcal{F}}_n) = \theta (1 + c (\varepsilon)), \]
  \[ \forall n \in \mathbb{N}, \quad \mathbb{E} ((\mathd X^{(i), j}_{n + 1})
     (\mathd X^{(i), k}_{n + 1}) | \widetilde{\mathcal{F}}_n) = \theta c
     (\varepsilon), \quad \tmop{for} j \neq k. \]
  Finally for all $p \geqslant 2$, there holds for all coordinates $j \in [0,
  d]$
  \[ \mathbb{E} (| \mathd X_{n + 1}^{(i), j} |^p | \widetilde{\mathcal{F}}_n)
     \lesssim \theta^{p / 2} . \]
\end{lemma}

\begin{proof}[discrete moments]
  Let $T > 0$ and $N \in \mathbb{N}$. Recall that the discretisation
  parameters $\delta$, $\theta$, and $\varepsilon$ are defined as
  \[ T \backassign N^5 \delta, \quad \theta \assign N \delta, \quad T \assign
     N^4 \theta, \quad \varepsilon \assign 1 / N. \]
  The increment
  \[ \mathd X_{n + 1}^{(i)} \assign X^{(i)}_{n + 1} - X_n^{(i)} = \sum_{k =
     1}^N \mathd B^{(i)}_{n N + k} \]
  involves $N$ consecutive steps of $B^{(i)}$, therefore $N$ layers of the
  dyadic tree, hence $N d$ generations. The layers involved are $\cup_{k =
  1}^N \mathcal{L}_{n N + k}$. Precisely,
  \[ \begin{array}{ll}
       \mathd X_{n + 1}^{(i)} (x) & = \sqrt{2 \delta}  \left( \sum_{k = 1}^N
       \varepsilon_{ (n N + k - 1) d + i}^+ (x), \sum_{k = 1}^N
       \varepsilon_{(n N + k - 1) d + 1}^- (x), \right.\\
       & \left. \sum_{k = 1}^N \varepsilon_{(n N + k - 1) d + 2}^- (x),
       \ldots, \sum_{k = 1}^N \varepsilon_{(n N + k - 1) d + d}^- (x) \right)
       .
     \end{array} \]
  we have immediately $\mathbb{E} (\mathd X_{n + 1}^{(i)} | \mathcal{F}_{n N})
  = (0, \ldots, 0)$. For the coordinate $j = 0$, we calculate the variance
  \[ \  \]
  \[ \begin{array}{lll}
       \mathbb{V} (\mathd X_{n + 1}^{(i), 0} | \widetilde{\mathcal{F}}_n) & =
       & 2 \delta \mathbb{V} \left( \sum_{k = 1}^N \varepsilon_{(n N + k - 1)
       d + j}^+ | \widetilde{\mathcal{F}}_n \right)\\
       & = & 2 \delta (\tmmathbf{1} (\varepsilon_{(n N + k - 1) d} = + 1))^2
       + 2 \delta \sum_{k = 1}^{N - 1} \frac{1}{2} [0^2 + 1^2]\\
       & = & \frac{2 \delta N}{2} (1 + c (\varepsilon)) = \theta (1 + c
       (\varepsilon))
     \end{array} \]
  and similarly for $j = 1$,
  \[ \mathbb{V} (\mathd X_{n + 1}^{(i), 1} | \widetilde{\mathcal{F}}_n) =
     \theta (1 + c (\varepsilon)) . \]
  Now for $2 \leqslant j \leqslant d$, we have
  \[ \mathbb{V} (\mathd X_{n + 1}^{(i), j} | \widetilde{\mathcal{F}}_n) = 2
     \delta \mathbb{V} \left( \sum_{k = 1}^N \varepsilon_{(n N + k - 1) d +
     j}^{\pm} | \widetilde{\mathcal{F}}_n \right) = 2 \delta \sum_{k = 1}^N
     \frac{1}{2} [0^2 + 1^2] = \theta . \]
  Regarding the covariation of the different coordinates of $\mathd X_{n +
  1}^{(i)}$, we have simply
  \[ \forall \alpha, \beta \in [0, d], \alpha \neq \beta, \quad \mathbb{E}
     (\mathd X_{n + 1}^{(i), \alpha} \mathd X_{n + 1}^{(i), \beta} |
     \widetilde{\mathcal{F}}_n) = 0. \]
  Indeed, increments of different coordinates involve tosses (i.e. Haar
  functions) with either pairwise--disjoint or included supports.
  
  The last estimate for moments of order $p \geqslant 2$ is first proved for
  even values of $p$ and then extended using H{\"o}lder to all values of $p
  \geqslant 2$. See {\cite{DomPet2023a}}.
\end{proof}

\section{Auxiliary convergence results}\label{S: auxiliary convergence
results}

The following two auxiliary convergence results are needed.

\begin{lemma}
  \label{L: weak convergence} (Weak convergence) Let $y > 0$ fixed. Let $i \in
  [1, d]$ fixed. Let $T > 0$. Let $\psi$ smooth on $\Omega$. Assume weak
  consistency. Then we have
  \[ \mathbb{E} \psi (X_T^{(i), \tau_{\varepsilon}}) =\mathbb{E} \psi (W_T^{y,
     \tau}) + c_{\psi, T} (\varepsilon) . \]
\end{lemma}

\begin{proof}
  This is a consequence of the weak consistency from Lemma \ref{L: weak
  consistency}. See e.g. {\cite{Tal1986n,KloPla1992}} where a slightly
  different definition of weak consistency is used. We sketch the proof
  adapted to our situation and notations. The function defined for all $t \in
  [0, T]$, $x \in \Omega$ as
  \[ u (t, x) \assign \mathbb{E} (\psi (W^{y, \tau}_T)  | W^{y, \tau}_t = x),
  \]
  is a solution to the backward heat equation
  \[ \partial_t u + \frac{1}{2} \Delta u = 0, \quad u (T, x) = \psi (x), \]
  and Ito formula reads
  \[ \psi (W^{y, \tau}_T) = u (T, W^{y, \tau}_T) = u (0, y) + \int_0^T \nabla
     u (\psi (W^{y, \tau}_s)) \cdot \mathd W^{y, \tau}_s + \int_0^T \left(
     \partial_t u + \frac{1}{2} \Delta u \right) (s, W^{y, \tau}_s) \mathd s,
  \]
  where the last term is null. Taking expectation allows to compare with the
  discrete process
  \[ \mathbb{E} \psi (X_T^{(i), \tau_{\varepsilon}}) -\mathbb{E} \psi (W_T^{y,
     \tau}) =\mathbb{E}u (T, X_T^{(i), \tau_{\varepsilon}}) - u (0, y)
     =\mathbb{E} [u (T, X_T^{(i), \tau_{\varepsilon}}) - u (0, X_0^{(i),
     \tau_{\varepsilon}})] . \]
  The last term is written as a telescopic sum, with summands
  \begin{eqnarray*}
    \mathbb{E} [u (t_{n + 1}, X_{t_{n + 1}}^{(i), \tau_{\varepsilon}}) - u
    (t_n, X_{t_n}^{(i), \tau_{\varepsilon}})] & = & \mathbb{E}u (t_{n + 1},
    X_{t_{n + 1}}^{(i), \tau_{\varepsilon}}) -\mathbb{E}u \left( t_{n + 1},
    W^{t_n, X_{t_n}}_{t_{n + 1}} \right)\\
    & = & \mathbb{E} \{ u (t_{n + 1}, X_{t_{n + 1}}^{(i),
    \tau_{\varepsilon}}) - u (t_n, X_{t_n}) \}\\
    &  & -\mathbb{E} \left\{ u \left( t_{n + 1}, W^{t_n, X_{t_n}}_{t_{n + 1}}
    \right) - u (t_n, X_{t_n}) \right\},
  \end{eqnarray*}
  where we have used again Ito formula to replace $u (t_n, X_{t_n}^{(i),
  \tau_{\varepsilon}})$ by $\mathbb{E}u \left( t_{n + 1}, W^{t_n,
  X_{t_n}}_{t_{n + 1}} \right)$. A Taylor expansion of the last two terms
  around the common point $(t_n, X_{t_n})$ exhibits the differences of
  increments $\{ X_{t_{n + 1}} - X_{t_n} \} - \left\{ W^{t_n, X_{t_n}}_{t_{n +
  1}} - W^{t_n, X_{t_n}}_{t_n} \right\}$ and higher moments. Those are
  estimated thanks to weak consistency as in the onedimensional case, yielding
  the desired result.
\end{proof}

\begin{lemma}
  \label{L: convergence discrete martingale transforms}(Weak convergence of
  discrete martingale transforms) Let $y > 0$ fixed. Let $i \in [1, d]$ fixed.
  Let $T > 0$. Let $f$ harmonic as above.
  \[ \| f (X_T^{(i)}) - M_T^{(i), f} \|_p = c_T (\varepsilon) \]
\end{lemma}

\begin{proof}
  The estimates are handeled slightly differently as in the onedimensional
  case in order to take advantage of the multidimensional harmonicity of $f$.
  Recalling that
  \[ \begin{array}{lll}
       M_T^{(i), f} - M_0^{(i), f} & \assign & \sum_{k = 1}^{N^5} \sum_{j = 1,
       d} \partial_j f (B^{(i)}_{k - 1}) \mathd B_k^{(i), j},
     \end{array} \]
  we compare with the Taylor expansion
  \begin{eqnarray*}
    f (X_T^{(i)}) - f (X_0^{(i)}) & = & f (B_T^{(i)}) - f (B^{(i)}_0) =
    \sum_{k = 1}^{N^5} [f (B_k^{(i)}) - f (B_{k - 1}^{(i)})]\\
    & = & \sum_{k = 1}^{N^5} \sum_{j = 1, d} \partial_j f (B^{(i)}_{k - 1})
    \mathd B_k^{(i), j} + \frac{1}{2} \sum_{k = 1}^{N^5} \sum_{j_1, j_2 = 1,
    d} \partial^2_{j_1 j_2} f (B^{(i)}_{k - 1}) \mathd B^{(i), j_1}_k \mathd
    B^{(i), j_2}_k\\
    &  & + \sum_{k = 1}^{N^5} R_3^f (B^{(i)}_{k - 1}, \mathd B^{(i)}_k) .\\
    & = & \sum_{k = 1}^{N^5} \sum_{j = 1, d} \partial_j f (B^{(i)}_{k - 1})
    \mathd B_k^{(i), j} + \frac{1}{2} \sum_{k = 1}^{N^5} \sum_{j = 1, d}
    \partial^2_j f (B^{(i)}_{k - 1}) (\mathd B^{(i), j}_k)^2\\
    &  & + \sum_{k = 1}^{N^5} R_3^f (B^{(i)}_{k - 1}, \mathd B^{(i)}_k) .
  \end{eqnarray*}
  where the last terms are Taylor remainders. Recalling that $M_0^{(i), f} = f
  (B^{(i)}_0)$, it remains to estimate
  \[ f (X_T^{(i)}) - M_T^{(i), f} = \frac{1}{2} \sum_{k = 1}^{N^5} \sum_{j_1,
     j_2 = 1, d} \partial^2_{j_1 j_2} f (B^{(i)}_{k - 1}) \mathd B^{(i),
     j_1}_k \mathd B^{(i), j_2}_k + \sum_{k = 1}^{N^5} R_3^f (B^{(i)}_{k - 1},
     \mathd B^{(i)}_k) \backassign A + B \]
  We have $| R_3^f (B_{k - 1} (x), \mathd B_k (x)) | \lesssim \| D^3 f
  \|_{\infty} \delta^{3 / 2}$ therefore $\| B \|_p \lesssim N^5 \delta^{3 / 2}
  = c_T (\varepsilon)$.
  
  Split now $A = \sum_{n = 1}^{N^4} A_n$ in $N^4$ blocks $(A_n)_{n = 1, N^4}$
  of $N$ terms each, and split further $A_n = B_n + C_n$ with
  \[ B_n \assign \sum_{l = 1}^N \sum_{j = 1}^d [\partial^2_j f (B^{(i)}_{(n -
     1) N + l - 1}) - \partial^2_j f (B^{(i)}_{(n - 1) N})] (\mathd B^{(i),
     j}_{(n - 1) N + l})^2, \]
  \[ C_n \assign \sum_{j = 1}^d \partial^2_j f (B^{(i)}_{(n - 1) N}) \sum_{l =
     1}^N (\mathd B^{(i), j}_{(n - 1) N + l})^2 . \]
  We estimate first $| B_n | \lesssim N \| D^3 f \|_{\infty} \left( N
  \sqrt{\delta} \right) \delta \lesssim N^2 \delta^{3 / 2}$. Therefore
  \[ \left\| \sum_1^{N^4} B_n \right\|_p \lesssim N^4 (N^2 \delta^{3 / 2}) =
     c_T (\varepsilon) . \]
  Recalling that $f$ is harmonic we can write equivalently
  \[ C_n \assign \sum_{j = 1}^d \partial^2_j f (B^{(i)}_{(n - 1) N}) \left[
     \sum_{l = 1}^N (\mathd B^{(i), j}_{(n - 1) N + l})^2 - \theta \right] .
  \]
  Notice that in the multidimensional case, we have that for a given
  {\tmem{fixed}} $j$, successive increments $\mathd B^{(i), j}_k$ and $\mathd
  B^{(i), j}_m$, $m \neq k$ are independent random variables, and the same
  holds for the sequence $(\mathd B^{(i), j}_{(n - 1) N + l})^2$, $1 \leqslant
  l \leqslant N$, with $\mathbb{E} (\mathd B^{(i), j}_{(n - 1) N + l})^2 =
  \delta$. Hence
  \[ \sum_{l = 1}^N (\mathd B^{(i), j}_{(n - 1) N + l})^2 - \theta = \sum_{l =
     1}^N [(\mathd B^{(i), j}_{(n - 1) N + l})^2 - \delta] \backassign \delta
     \sum_{l = 1}^N \alpha_{(n - 1) N + l - 1}, \]
  where $\alpha_k = \pm 1$, since $(\mathd B^{(i), j}_{(n - 1) N + l})^2 = 0$
  or $(\mathd B^{(i), j}_{(n - 1) N + l})^2 = 2 \delta$ with equal
  probability. In other words $\sum_{l = 1}^N \alpha_{(n - 1) N + l - 1}$ is a
  discrete $N$ step random walk with $\pm 1$ steps. It follows, recalling
  moments of discrete random walks, that
  \[ \| C_n \|_p \lesssim \| D^2 f \|_{\infty} \delta \left\| \sum_{l = 1}^N
     \alpha_{(n - 1) N + l - 1} \right\|_p \lesssim \| D^2 f \|_{\infty}
     \delta (\theta^{p / 2})^{1 / p} . \]
  Finally
  \[ \left\| \sum_{n = 1}^{N^4} C_n \right\|_p \lesssim N^4 \delta \| D^2 f
     \|_{\infty} (N \delta)^{1 / 2} = (N^5 \delta) \| D^2 f \|_{\infty} N^{- 1
     / 2} \delta^{1 / 2} = c_T (\varepsilon) . \]
  This concludes the proof of Lemma \ref{L: convergence discrete martingale
  transforms}.
\end{proof}

\

\

\section{Proof of the main results}

\subsection{Upper bound for a single Riesz transform}

\begin{lemma}[Convergence of martingales in $L^p$]
  \label{L: convergence}Let $y > 0$ fixed. Let $i \in [1, d]$ fixed. Let $f
  \in L^p (\partial \Omega)$ and harmonic in $\Omega$. We have
  \[ \lim_{T \rightarrow \infty} \lim_{\varepsilon \rightarrow 0} \mathbb{E} |
     M_T^f |^p =\mathbb{E} | f (W^{y, \tau}_{\infty}) |^p, \]
\end{lemma}

\begin{proof}
  The proof is the same as the onedimensional case {\cite{DomPet2023a}}. It is
  merely a combination of \ Lemmas \ref{L: weak convergence} \& \ref{L:
  convergence discrete martingale transforms}, where we compare successively
  $f (W^{y, \tau}_T)$ to $f (X^{(i)}_T)$, and $f (X^{(i)}_T)$to \ $M^{(i),
  f}_T$.
\end{proof}

\begin{lemma}[Convergence in the weak formulations]
  \label{L: convergence weak formulation}Let $y > 0$ fixed. Let $i \in [1, d]$
  fixed. Let $f \in L^p (\partial \Omega, X)$ and $g \in L^p (\partial \Omega,
  X^{\ast})$ both harmonic. We have
  \[ \lim_{T \rightarrow \infty} \lim_{\varepsilon \rightarrow 0} \mathbb{E}
     \langle M_T^{(i), f}, M_T^{(i), g} \rangle =\mathbb{E} \langle
     \mathcal{M}^{y, f}_{\infty}, \mathcal{M}^{y, g}_{\infty} \rangle . \]
  \[ \lim_{T \rightarrow \infty} \lim_{\varepsilon \rightarrow 0} \mathbb{E}
     \langle M_T^{(i), i}, M_T^{(i), g} \rangle =\mathbb{E} \langle
     \mathcal{M}^{y, i}_{\infty}, \mathcal{M}^{y, g}_{\infty} \rangle . \]
\end{lemma}

\begin{proof}
  For the first limit we can use the same strategy as for the proof of Lemma
  \ref{L: convergence}. This is because the first limit involves the
  discretisation of the harmonic function $f$. However the second limit
  involves the martingale transform $M_T^{(i), i}$ of $M_T^{(i), f}$, or in
  the limit the martingale transform $\mathcal{M}^{y, i}_{\infty}$ of
  $\mathcal{M}^{y, f}_{\infty}$. Those are not martingales directly associated
  to $\mathcal{R}_i f$ or its discretisation, since a projection
  \eqref{eq:conditonal expectation Ri} is needed to recover $\mathcal{R}_i f$.
  
  \
  
  Indeed for the first limit, using Lemma \ref{L: convergence discrete
  martingale transforms} and recalling that all functions are smooth with
  compact support, we have successively
  \[ \mathbb{E} \langle M_T^{(i), f}, M_T^{(i), g} \rangle =\mathbb{E} \langle
     f (X_T^{(i), \tau_{\varepsilon}}), M_T^{(i), g} \rangle + c_T
     (\varepsilon), \]
  and similarly
  \[ \mathbb{E} \langle f (X_T^{(i), \tau_{\varepsilon}}), M_T^{(i), g}
     \rangle =\mathbb{E} \langle f (X_T^{(i), \tau_{\varepsilon}}), g
     (X_T^{(i), \tau_{\varepsilon}}) \rangle + c_T (\varepsilon) . \]
  Applying the weak convergence of Lemma \ref{L: weak convergence} to $\psi
  (x) \assign \langle f (x), g (x) \rangle$ then letting $T$ go to infinity
  yields the first limit.
  
  \
  
  For the second limit, explicit functions of $X^{(i)}$ are only accessible
  after expanding the product martingales. Indeed, given $T > 0$, we have
  \[ \mathbb{E} \langle \mathcal{M}^{y, i}_T, \mathcal{M}^{y, g}_T \rangle
     =\mathbb{E} \langle \mathcal{M}^{y, i}_0, \mathcal{M}^{y, g}_0 \rangle
     +\mathbb{E} \int_0^T \langle A_i \nabla f (W^{y, \tau}_t), \nabla g
     (W^{y, \tau}_t) \rangle \mathd t, \]
  and similarly
  \[ \mathbb{E} \langle M_T^{(i), i}, M_T^{(i), g} \rangle =\mathbb{E} \langle
     M_0^{(i), i}, M_0^{(i), g} \rangle +\mathbb{E} \sum_{n = 1}^{N^4} \langle
     A_i \nabla f (X_{t_n}), \nabla g (X_{t_n}) \rangle  (t_{n + 1} - t_n) .
  \]
  Notice that the first terms on the right hand sides are null. Choose now a
  large integer $K$. We know that we can approximate the continuous stochastic
  integral as
  \[ \mathbb{E} \int_0^T \langle A_i \nabla f (W^{y, \tau}_t), \nabla g (W^{y,
     \tau}_t) \rangle \mathd t =\mathbb{E} \sum_{k = 1}^K \langle A_i \nabla f
     (W^{y, \tau}_{T_{k - 1}}), \nabla g (W^{y, \tau}_{T_{k - 1}}) \rangle 
     (T_k - T_{k - 1}) + c (K), \]
  where $(T_k)_{k \in [0, K]}$ is a regular subdivision of $[0, T]$, i.e. $T_k
  = k T / K$, and $c (K)$ is a function that goes to zero when $K$ goes to
  infinity. Similarly, it is now an exercise to show that we can approximate
  \[ \mathbb{E} \sum_{n = 1}^{N^4} \langle A_i \nabla f (X_{t_n}), \nabla g
     (X_{t_n}) \rangle  (t_{n + 1} - t_n) =\mathbb{E} \sum_{k = 1}^K \langle
     A_i \nabla f (X_{T_{k - 1}}), \nabla g (X_{T_{k - 1}}) \rangle  (T_k -
     T_{k - 1}) + c (K), \]
  where we assume without loss of generality that $N$ is a multiple of $K$. We
  can now pass to the limit for each $k \in [1, K]$ so that $\mathbb{E}
  \langle A_i \nabla f (X_{T_{k - 1}}), \nabla g (X_{T_{k - 1}}) \rangle
  \rightarrow \mathbb{E} \langle A_i \nabla f (W^{y, \tau}_{T_{k - 1}}),
  \nabla g (W^{y, \tau}_{T_{k - 1}}) \rangle$ when $\varepsilon$ goes to zero,
  using again weak convergence. This allows to prove the second limit,
  therefore concludingg the proof of Lemma \ref{L: convergence weak
  formulation}.
\end{proof}

Finally, we prove our main result.

\begin{proof*}{Proof of Theorem \ref{T: upper bound single Riesz}}
  {\dueto{Upper bound for a single riesz transform}}Recall that we know from
  Lemma \ref{L: Lp estimate for Mii} that
  \[ \| M_k^{(i), i} \|_p \leqslant \| \mathcal{S}_i \|_{p \rightarrow p}  \|
     M^{(i), f}_k \|_p . \]
  We have equivalently
  \[ \mathbb{E} \langle M_k^{(i), i}, M^{(i), g} \rangle \leqslant \|
     \mathcal{S}_i \|_{p \rightarrow p}  \| M^{(i), f}_k \|_p \| M^{(i), g}_k
     \|_p \]
  for any test function $g \in L^q (\partial \Omega, X^{\ast})$. Using Lemmas
  \ref{L: convergence} \& \ref{L: convergence weak formulation}, we pass to
  the limit $T \rightarrow \infty$ and $\varepsilon \rightarrow 0$ and obtain,
  for any fixed $y > 0$,
  \[ \mathbb{E} \langle \mathcal{M}^{y, i}_{\infty}, \mathcal{M}^{y,
     g}_{\infty} \rangle \leqslant \| \mathcal{S}_i \|_{p \rightarrow p} 
     (\mathbb{E} | f (W^{y, \tau}_{\infty}) |^p)^{1 / p}  (\mathbb{E} | g
     (W^{y, \tau}_{\infty}) |^q)^{1 / q} \]
  Notice that
  \[ \mathbb{E} | f (W^{y, \tau}_{\infty}) |^p = \int_{\partial \Omega} | f
     (x) |^p \mathd \omega^y (x) \]
  where $\mathd \omega^y$ is the harmonic measure associated to the Brownian
  $W^{y, \tau}$ started at $y$ and stopped upon hitting $\partial \Omega$. Due
  to translation invariance of the domain, we have $\mathd \omega^y (x) =
  \mathd \omega^0 (x - y) = G (x - y) \mathd x$, where $G$ is the normalized
  Gaussian. It follows that
  \[ \int_y \mathbb{E} | f (W^{y, \tau}_{\infty}) |^p = \int_y \int_{\partial
     \Omega} | f (x) |^p \mathd \omega^y (x) = \left( \int_{\partial \Omega} |
     f (x) |^p \mathd x \right) \left( \int_y G (y) \mathd y \right) = \| f
     \|_p, \]
  hence
  \[ \int_y (\mathbb{E} | f (W^{y, \tau}_{\infty}) |^p)^{1 / p}  (\mathbb{E} |
     g (W^{y, \tau}_{\infty}) |^q)^{1 / q} \leqslant \| f \|_p  \| g \|_q . \]
  Following {\cite{GunVar1979a}}, we observe that
  \[ \lim_{y \rightarrow \infty} \int_y \mathbb{E} \langle \mathcal{M}^{y,
     i}_{\infty}, \mathcal{M}^{y, g}_{\infty} \rangle = \int_{\Omega} \langle
     A_i \nabla f, \nabla g \rangle 2 y \mathd y \mathd x = \langle
     \mathcal{R}_i f, g \rangle, \]
  yielding
  \[ \langle \mathcal{R}_i f, g \rangle \leqslant \| \mathcal{S}_i \|_{p
     \rightarrow p} \| f \|_p  \| g \|_q, \]
  hence the desired result.
\end{proof*}

\subsection{Upper bound for the vector Riesz transform}

Let $\mathcal{R} f \assign \left( \sum_{i = 1}^d | \mathcal{R}_i f |_X^2
\right)^{1 / 2}$ and $\mathcal{S} f \assign \left( \sum_{i = 1}^d |
\mathcal{S}_i f |_X^2 \right)^{1 / 2}$. Our goal is to prove Theorem \ref{T:
upper bound Riesz vector}, namely
\[ \| \mathcal{R} \|_p \leqslant \| \mathcal{S} \|_p . \]
Let $g = (g_i)_{i = 1, \ldots, d}$ a vector consisting of $d$ test functions.
The usual weak formulation for the estimate of the Riesz vector reads
\[ \begin{array}{rcl}
     \| \mathcal{R} f \|_p & = & \sup_{\| g \|_q \leqslant 1} \langle
     \mathcal{R} f, g \rangle = \sup_{\| g \|_q \leqslant 1} \sum_{i = 1}^d
     \langle \mathcal{R}_i f, g_i \rangle = \sup_{\| g \|_q \leqslant 1}
     \mathbb{E} \sum_{i = 1}^d \int_0^{\tau} \langle A_i \nabla f, \nabla g_i
     \rangle \mathd t
   \end{array} \]
The last term is the stochastic representation involving the $(d +
1)$--dimensional Brownian motion $W^{\tau}$, namely
\[ \langle \mathcal{R} f, g \rangle =\mathbb{E} \sum_{i = 1}^d \int_0^{\tau}
   \langle A_i \nabla f, \nabla g_i \rangle =\mathbb{E} \sum_{i = 1}^d
   \int_0^{\tau} A_i \nabla f (W_s^{\tau}) \cdot \nabla g_i (W_s^{\tau})
   \mathd t. \]
Notice that all summands involve the \tmtextit{same} realization of the
Brownian motion.

\

\paragraph{Adapted stochastic representation}In order for us to compare the
dyadic Riesz vector and the usual Riesz vector, we need to adapt the
stochastic representation in a simple (but effective!) way. Observing that
\[ \langle \mathcal{R} f, g \rangle =\mathbb{E} \sum_{i = 1}^d \int_0^{\tau}
   A_i \nabla f (W_s^{\tau}) \cdot \nabla g_i (W_s^{\tau}) \mathd t = \sum_{i
   = 1}^d \mathbb{E} \int_0^{\tau} A_i \nabla f (W_s^{\tau}) \cdot \nabla g_i
   (W_s^{\tau}) \mathd t, \]
one realizes that the expectation operator $\mathbb{E}$ is applied $d$ times
in the last expression, and that the Brownian motion is a dummy variable. We
can therefore go further and reformulate
\begin{equation}
  \begin{array}{lll}
    \langle \mathcal{R} f, g \rangle & = & \sum_{i = 1}^d \mathbb{E}^i
    \int_0^{\tau_i} A_i \nabla f (W_s^{(i), \tau_i}) \cdot \nabla g_i
    (W_s^{(i), \tau_i}) \mathd t,\\
    & = & \mathbb{E} \sum_{i = 1}^d \int_0^{\tau_i} A_i \nabla f (W_s^{(i),
    \tau_i}) \cdot \nabla g_i (W_s^{(i), \tau_i}) \mathd t,
  \end{array} \label{eq: adapted stochastic representation}
\end{equation}
where the $W^{(i), \tau_i}$'s are $d$ different realizations of Brownian
motions in $\mathbb{R}^{d + 1}$, $\mathbb{E}^i$ the corresponding expectations
and $\tau_i$ the corresponding stopping times. Those realizations do
\tmtextit{not} need to be independent. This is reminiscent of the $d$ discrete
random walks $B^{(i)}$, each stopped at different times but sharing the same
horizontal component.

\begin{proof*}{Proof of Theorem \ref{T: upper bound Riesz vector}}
  {\dueto{Upper bound for the Riesz vector}}For all $i \in [1, d]$, we
  approximate $\mathcal{M}_t^f$ as well as the test function
  $\mathcal{M}_t^{g_i}$, by
  \[ M^{(i), f}_k \assign \mathcal{M}_0^f + \sum_{\ell = 1}^k \nabla f (B_{k -
     1}^{s (i)}) \cdot \mathd B_k^{(i)}, \quad M^{(i), g_i}_k \assign
     \mathcal{M}_0^{g_i} + \sum_{\ell = 1}^k \nabla g_i (B_{k - 1}^{(i)})
     \cdot \mathd B_k^{(i)} . \]
  We have by construction, for any fixed $y$, any fixed numerical parameters,
  that
  \[ \sum_{i = 1}^d \mathbb{E}^i M^{(i), \mathcal{S}_i f}_T M^{(i), g_i}_T =
     \sum_{i = 1}^d \langle \mathcal{S}_i f, g_i \rangle \leqslant \|
     \mathcal{S} \|_{p \rightarrow p} \| f \|_p \| g \|_q . \]
  An immediate consequence of Lemma \ref{L: convergence weak formulation}
  reads
  \[ \lim_{T \rightarrow \infty} \lim_{\varepsilon \rightarrow 0} \sum_{i =
     1}^d \mathbb{E}^i M^{(i), f}_T M^{(i), g_i}_T = \sum_{i = 1}^d
     \mathbb{E}^i  \mathcal{M}^{(i), f}_{\infty} M^{(i), g_i}_{\infty} \]
  \[ \lim_{T \rightarrow \infty} \lim_{\varepsilon \rightarrow 0} \sum_{i =
     1}^d \mathbb{E}^i M^{(i), \mathcal{S}_i f}_T M^{(i), g_i}_T = \sum_{i =
     1}^d \mathbb{E}^i  \mathcal{M}^{(i), \mathcal{R}_i f}_{\infty} M^{(i),
     g_i}_{\infty} \]
  Passing to the limit as for the proof of Theorem \ref{T: upper bound single
  Riesz} yields
  \[ \langle \mathcal{R} f, g \rangle \leqslant \| \mathcal{S} \|_{p
     \rightarrow p} \| f \|_p \| g \|_q, \]
  hence the desired result $\| \mathcal{R} \|_p \leqslant \| \mathcal{S}
  \|_p$.
\end{proof*}

\end{document}